\documentclass[10pt]{amsart}
\usepackage{a4}
\usepackage{amssymb}
\usepackage{array}
\usepackage{booktabs}
\usepackage{multirow}
\usepackage{hhline}
\usepackage{mathabx}

\newtheorem{theorem}{Theorem}

\newtheorem{corollary}[theorem]{Corollary}
\newtheorem{lemma}[theorem]{Lemma}
\newtheorem{proposition}[theorem]{Proposition}
\newtheorem{remark}[theorem]{Remark}

\numberwithin{theorem}{section}
\numberwithin{equation}{section}

\newcommand{\g}{\mathbb{G}}

\newcommand{\n}{\mathbb{N}}

\begin{document}
\title{A theorem for random Fourier series on compact quantum groups}

\author{Sang-Gyun Youn}

\address{Sang-Gyun Youn : Department of Mathematical Sciences, Seoul National University,
San56-1 Shinrim-dong Kwanak-gu, Seoul 151-747, Republic of Korea}
\email{yun87654@snu.ac.kr}

\keywords{Random Fourier series, compact quantum group, non self-adjoint operator algebra}
\thanks{2010 \it{Mathematics Subject Classification}.
\rm{Primary 46L89, Secondary 20G42, 43A30}.\\ The author is supported by TJ Park Science Fellowship}

\begin{abstract}
Helgason showed that a given measure $f\in M(G)$ on a compact group $G$ should be in $L^2(G)$ automatically if all random Fourier series of $f$ are in $M(G)$. We explore a natural analogue of the theorem in the framework of compact quantum groups and apply the obtained results to study complete representability problem for convolution algebras of compact quantum groups as an operator algebra. 
\end{abstract}

\maketitle

\section{Introduction}

The random Fourier series theory has been extensively studied for a long period of time \cite{PaZy32}, \cite{Ka93}, \cite{SaZy54}, \cite{Hu51}, \cite{Du67}, \cite{Fe75}, \cite{Ma78} and \cite{MaPi78} and most of the results turned out to be valid for Banach space-valued functions on compact groups  \cite{FiRi66}, \cite{Kw76} and \cite{MaPi81}. Moreover, it has been found that these studies can be applied to various topics of harmonic analysis \cite{Pi77}, \cite{Pi78} and \cite{Ri75}.

A theorem of Helgason \cite{He57}, which is a subject of this paper, is one of the results of the theory of random Fourier series. It is an improvement of the Littlewood's work on circle \cite{Li26}.

Throughout this paper, we denote by $\mathrm{Irr}(G)$ a maximal family of mutually inequivalent irreducible unitary representations of a compact group $G$ and by $\mathcal{U}$ the product group $\displaystyle \prod_{\pi\in \mathrm{Irr}(G)}U(n_{\pi})$ where $U(d)$ is the group of unitary matrices of size $d$ and $n_{\pi}$ is the dimension of $\pi\in \mathrm{Irr}(G)$. Also, we denote by 
\[f\sim \sum_{\pi\in \mathrm{Irr}(G)}n_{\pi}\mathrm{tr}(\widehat{f}(\pi)\pi(x))\]
the Fourier-Stieltjes series of $f\in M(G)$, where $(\widehat{f}(\pi))_{\pi\in \mathrm{Irr}(G)}$ is the sequence of Fourier coefficients of $f\in M(G)$.

\begin{theorem}[S.Helgason, \cite{He57}]\label{thm1}

Let $G$ be a compact group and fix a measure $f\in M(G)$. Suppose that there exists $\mu_U\in M(G)$ such that
$\displaystyle \mu_U\sim \sum_{\pi\in \widehat{G}}n_{\pi}\mathrm{tr}(U_{\pi}\widehat{f}(\pi)\pi(x))$ for any $U=(U_{\pi})_{\pi\in \widehat{G}}\in \mathcal{U}$. Then 
\[\sum_{\pi\in \widehat{G}}n_{\pi}\mathrm{tr}(\widehat{f}(\pi)^*\widehat{f}(\pi))<\infty.\]

\end{theorem}

After the works of S.Woronowicz, which are \cite{Wo87b}, \cite{Wo87a} for compact quantum groups, the theory of topological quantum groups has been greatly elevated \cite{KuVa00}, \cite{Ku01}, \cite{KuVa03} and harmonic analysis on quantum groups has also been vigorously studied in \cite{Ca13}, \cite{FHLUZ17}, \cite{JMP14}, \cite{JPPP17} and \cite{Yo17}. In particular, progress has been made in research on random Fourier series \cite{Wa17}. The main purpose of this paper is to generalize Theorem \ref{thm1} in the framework of compact quantum groups in a natural manner through mobilizing well-known studies on stochastic behavior of some special vector-valued random variables(e.g. Rademacher and Gaussian variables), namely the non-commutative Khintchine inequalities \cite{Lu86}, \cite{LuPi91} and (co-)type studies \cite{To74}, \cite{Fa87}.

The Fourier analysis on compact quantum groups has been developed in respect of traditional philosophy of the Fourier analysis on compact groups to large extent. However, a visible difference appears in the Fourier expansion. In fact, our random Fourier series of $f\in M(\g)$ will be described as
\[ f_U\sim  \sum_{\alpha\in \mathrm{Irr}(\g)}d_{\alpha}\mathrm{tr}(U_{\alpha}\widehat{f}(\alpha)Q_{\alpha}u^{\alpha})\]
for any $U=(U_{\alpha})_{\alpha\in \mathrm{Irr}(\g)}\in \mathcal{U}$. We will explain the details of the above expansion in Section \ref{sec:pre}.

As in the classical setting, $f\in L^2(\g)$ will imply that all of random fourier series $f_U$ are in $L^2(\g)\subseteq L^1(\g)$ by the Plancherel identity on compact quantum groups (Proposition \ref{prop-easy}). Accordingly what we have to do is to demonstrate the converse direction and main results of this paper are as follows.

\begin{theorem}\label{thm3}
In the above notation, let us suppose that $f_U\in M(\g)$ for all $\displaystyle U=(U_{\alpha})_{\alpha\in \mathrm{Irr}(\g)}\in \mathcal{U}$.

\begin{enumerate}
\item If $\g$ is of Kac type, then we have
\[\sum_{\alpha\in \mathrm{Irr}(\g)}n_{\alpha}\mathrm{tr}(\widehat{f}(\alpha)^*\widehat{f}(\alpha))<\infty.\]
\item For general compact quantum groups, we have
\[\sum_{\alpha\in \mathrm{Irr}(\g)}\frac{d_{\alpha}}{n_{\alpha}}\mathrm{tr}(Q_{\alpha}\widehat{f}(\alpha)^*\widehat{f}(\alpha))<\infty.\]
\end{enumerate}
\end{theorem}

Moreover, we will apply the obtained results to a problem determining whether the associated convolution algebra $L^1(\g)$ can be completely isomorphic to a closed subalgebra of $B(H)$, in which $H$ is a Hilbert space (See Section \ref{sec:application}). In this case, we will say that $L^1(\g)$ is completely representable as an operator algebra.

\begin{corollary}
Let $\g=(A,\Delta)$ be a compact quantum group. Then $L^1(\g)$ is completely representable as an operator algebra if and only if $A$ is finite dimensional.
\end{corollary}

\section{Preliminaries}\label{sec:pre}
\subsection{Compact quantum group}\label{subsec:rep}

A compact quantum group is a pair $\g=(A,\Delta)$ where $A$ is a unital $C^*$-algebra and $\Delta:A\rightarrow A\otimes_{min} A$ is a unital $*-$homomorphism satisfying
\begin{enumerate}
\item $(\Delta\otimes \mathrm{id})\circ \Delta=(\mathrm{id} \otimes \Delta)\circ \Delta$ and
\item $\mathrm{span}\left \{ \Delta(a)(b\otimes 1):a,b\in A\right\}$ and $\mathrm{span}\left \{ \Delta(a)(1\otimes b):a,b\in A\right\}$ are dense in $A\otimes_{min}A.$
\end{enumerate}

Let us describe the basic representation theory and {\it the Schur's orthogonality relation} on compact quantum groups. A finite dimensional {\it unitary representation} of $\mathbb{G}$ is $u=(u_{i,j})_{1\leq i,j\leq n}\in M_n(A)$ such that $\Delta(u_{i,j})=\displaystyle \sum_{k=1}^n u_{i,k}\otimes u_{k,j}$ and $\displaystyle (u^*u)_{i,j}=\sum_{k=1}^nu_{k,i}^*u_{k,j}=\delta_{i,j}1_A=\sum_{k=1}^n u_{i,k}u^*_{j,k}=(uu^*)_{i,j}$ for all $1\leq i,j\leq n$. We say that a unitary representation $u\in M_n(A)$ is {\it irreducible} if 
$\left \{X\in M_n:Xu=uX\right\}=\mathbb{C}\cdot \mathrm{Id}_n$ and denote by $\left \{u^{\alpha}=(u^{\alpha}_{i,j})_{1\leq i,j\leq n_{\alpha}}\right\}_{\alpha\in \mathrm{Irr}(\g)}$ a maximal family of mutually inequivalent unitary irreducible representations of $\g$.

Every compact quantum group allows the analogue of the Haar measure, namely the {\it Haar state}, which is the unique state $h$ on $A$ such that $(\mathrm{id}\otimes h)(\Delta(a))=h(a)1_A=(h\otimes \mathrm{id})(\Delta(a))$ for all $a\in A$ and $h(1_A)=1$.

For each $\alpha\in \mathrm{Irr}(\mathbb{G})$, there exists a unique mathrix $Q_{\alpha}\in M_{n_{\alpha}}$ such that $\displaystyle Q_{\alpha}^{\frac{1}{2}}\overline{u^{\alpha}}Q_{\alpha}^{-\frac{1}{2}}$ is unitary and $\mathrm{tr}(Q_{\alpha})=\mathrm{tr}(Q_{\alpha}^{-1})$
where $\overline{u^{\alpha}}:=((u^{\alpha}_{i,j})^*)_{1\leq i,j\leq n_{\alpha}}$. We call $d_{\alpha}=\mathrm{tr}(Q_{\alpha})=\mathrm{tr}(Q_{\alpha}^{-1})$ {\it the quantum dimension} of $\alpha\in \mathrm{Irr}(\g)$.
 For any $\alpha,\beta\in \mathrm{Irr}(\g)$, $1\leq i,j\leq n_{\alpha}$ and $1\leq s,t\leq n_{\beta}$, we have the Schur's orthogonality relation as follows. 
\[h((u^{\beta}_{s,t})^*u^{\alpha}_{i,j})=\frac{\delta_{\alpha,\beta}\delta_{j,t}(Q_{\alpha}^{-1})_{i,s}}{\mathrm{tr}(Q_{\alpha})}~\mathrm{and~}h(u^{\beta}_{s,t}(u^{\alpha}_{i,j})^*)=\frac{\delta_{\alpha,\beta}\delta_{i,s}(Q_{\alpha})_{j,t}}{\mathrm{tr}(Q_{\alpha})}.\]

Moreover, we may assume that the matrices $Q_{\alpha}$ are diagonal \cite{Da10}. Also, we say that $\g$ is of {\it Kac type} if $Q_{\alpha}=\mathrm{Id}_{n_{\alpha}}$ for all $\alpha\in \mathrm{Irr}(\g)$ or equivalently if the Haar state $h$ is tracial.

We define the space of matrix coefficients as
\[\mathrm{Pol}(\g)=\mathrm{span}\left \{u^{\alpha}_{i,j}:\alpha\in \mathrm{Irr}(\mathbb{G}),1\leq i,j\leq n_{\alpha}\right\}\]
and then the Haar state $h$ is faithful on $\mathrm{Pol}(\g)$. We denote by $L^2(\g)$ the closure of $\mathrm{Pol}(\g)$ with respect to the pre-inner product $\langle a,b\rangle_{L^2(\g)}=h(b^*a)$ for all $a,b\in \mathrm{Pol}(\g)$ and denote by $\Lambda:\mathrm{Pol}(\g)\hookrightarrow L^2(\g),~a\mapsto \Lambda(a)$ the natural embedding. Also, we define $C_r(\g)$ as the norm-closure of $\mathrm{Pol}(\g)$ in $B(L^2(\g))$ under the GNS representation $[\iota(a)](\Lambda(x)):=\Lambda(ax)$ for all $a,x\in \mathrm{Pol}(\g)$ and $L^{\infty}(\g)$ as weak $*$-closure of $C_r(\g)$ in $B(L^2(\g))$.

The haar state $h$ extends to $L^{\infty}(\g)$ as a normal faithful state and we will denote it by $h$ again. We consider the predual $L^1(\g)$ of $L^{\infty}(\g)$ and the dual $M(\g)$ of $C_r(\g)$ respectively. Then we have a contractive embedding $L^{\infty}(\g)\hookrightarrow L^1(\g),~x\mapsto h(\cdot x)$ and the isometric formal identity from $L^1(\g)$ into $M(\g)$. Note that $\mathrm{Pol}(\g)$ is dense in $L^1(\g)$.

\subsection{Fourier analysis on compact quantum groups}

It is known that every compact quantum group admits the dual discrete quantum group $\widehat{\g}$ with a natural quantum group structure under {\it the generalized Pontrjagin duality}. Although we will not mention its quantum group structure, we will specifically explain some of non-commutative $\ell^p$-spaces on the discrete quantum group $\widehat{\g}$.

The associated von Neumann algebra of $\widehat{\g}$ is given by
\[\ell^{\infty}(\widehat{\g})=\left \{(X_{\alpha})_{\alpha\in \mathrm{Irr}(\g)}\in \prod_{\alpha\in \mathrm{Irr}(\g)}M_{n_{\alpha}}:\sup_{\alpha\in \mathrm{Irr}(\g)}\left \|X_{\alpha}\right \|<\infty\right\}\]
with the norm $\displaystyle \left \|(X_{\alpha})_{\alpha\in \mathrm{Irr}(\g)} \right \|_{\ell^{\infty}(\widehat{\g})}=\sup_{\alpha\in \mathrm{Irr}(\g)}\left \|X_{\alpha}\right \|_{M_{n_{\alpha}}}$.

We define the Fourier-Stieltjes transform by $\mathcal{F}:M(\g)=C_r(\g)^*\rightarrow \ell^{\infty}(\widehat{\g})$,
\[\mu\mapsto (\widehat{\mu}(\alpha))_{\alpha\in \mathrm{Irr}(\g)}~\mathrm{with~}\widehat{\mu}(\alpha)=(\mu((u^{\alpha}_{j,i})^*))_{1\leq i,j\leq n_{\alpha}},\]
to follow the standard notations of Fourier analysis on compact groups. Then the Fourier expansion of $\mu\in M(\g)$ is given by
\[\mu\sim \sum_{\alpha\in \mathrm{Irr}(\g)}d_{\alpha}\mathrm{tr}(\widehat{\mu}(\alpha)Q_{\alpha}u^{\alpha})=\sum_{\alpha\in \mathrm{Irr}(\g)}\sum_{i,j=1}^{n_{\alpha}}d_{\alpha}(\widehat{\mu}(\alpha)Q_{\alpha})_{i,j}u^{\alpha}_{j,i}.\]

Indeed, if $f\in \mathrm{Pol}(\g)$, then we have
\[f= \sum_{\alpha\in \mathrm{Irr}(\g)}d_{\alpha}\mathrm{tr}(\widehat{f}(\alpha)Q_{\alpha}u^{\alpha}).\]

The associated $\ell^2$ and $\ell^1$ spaces on $\widehat{\g}$ are
\begin{align*}
\ell^2(\widehat{\mathbb{G}})=&\left \{ (X_{\alpha})_{\alpha\in \mathrm{Irr}(\mathbb{G})}\in \ell^{\infty}(\mathbb{\widehat{G}}): \sum_{\alpha\in \mathrm{Irr}(\mathbb{G})}d_{\alpha}\mathrm{tr}(Q_{\alpha}X_{\alpha}^*X_{\alpha})<\infty \right\},\\
\ell^1(\widehat{\mathbb{G}})=&\left \{ (X_{\alpha})_{\alpha\in \mathrm{Irr}(\mathbb{G})}\in \ell^{\infty}(\mathbb{\widehat{G}}): \sum_{\alpha\in \mathrm{Irr}(\mathbb{G})}d_{\alpha}\mathrm{tr}(|X_{\alpha}Q_{\alpha}|)<\infty \right\}
\end{align*}
with the norm structures
\begin{align*}
\left \|(X_{\alpha})_{\alpha\in \mathrm{Irr}(\g)}\right \|_{\ell^2(\widehat{\g})}=&(\sum_{\alpha\in \mathrm{Irr}(\g)}d_{\alpha}\mathrm{tr}(Q_{\alpha}X_{\alpha}^*X_{\alpha}))^{\frac{1}{2}},\\
\left \|(X_{\alpha})_{\alpha\in \mathrm{Irr}(\g)}\right \|_{\ell^1(\widehat{\g})}=&\sum_{\alpha\in \mathrm{Irr}(\g)}d_{\alpha}\mathrm{tr}(|X_{\alpha}Q_{\alpha}|)
\end{align*}
respectively. When we restrict the domain of the Fourier-Stieltjes transform to $L^2(\g)$, we have the isometry $\mathcal{F}:L^2(\g)\rightarrow \ell^2(\widehat{\g})$ \cite{Wa17}, which is called the Plancherel identity.

\subsection{Random Fourier series}

From now on, we call formal series 
\[ \sum_{\alpha\in \mathrm{Irr}(\g)}d_{\alpha}\mathrm{tr}(U_{\alpha}\widehat{f}(\alpha)Q_{\alpha}u^{\alpha})\] 
the random Fourier series of $f\in M(\g)$ for $U=(U_{\alpha})_{\alpha\in \mathrm{Irr}(\g)}\in \displaystyle \mathcal{U}$. The main question of this paper is to find out when all of the random Fourier series are in $M(\g)$ simultaneously. It can be seen that one direction on our question can be solved simply from the Plancherel identity.

\begin{proposition}\label{prop-easy}
If $f\in L^2(\g)$, all of random fourier series $f_U$ are in $L^2(\g)$.
\end{proposition}

\begin{proof}

\begin{align*}
\left \|f_U\right \|_{L^2(\g)}^2&=\sum_{\alpha\in \mathrm{Irr}(\g)}d_{\alpha}\mathrm{tr}(Q_{\alpha}\widehat{f}(\alpha)^*U_{\alpha}^*U_{\alpha}\widehat{f}(\alpha))\\
&=\sum_{\alpha\in \mathrm{Irr}(\g)}d_{\alpha}\mathrm{tr}(Q_{\alpha}\widehat{f}(\alpha)^*\widehat{f}(\alpha))=\left \|f\right\|_{L^2(\g)}^2.
\end{align*}

\end{proof}

\begin{lemma}\label{lem1}

For $\displaystyle \mu\sim \sum_{\alpha\in \mathrm{Irr}(\g)}d_{\alpha}\mathrm{tr}(\widehat{\mu}(\alpha)Q_{\alpha}u^{\alpha})\in M(\mathbb{G})$
and $f\in \mathrm{Pol}(\mathbb{G})$,
\[\langle \mu,f^*\rangle_{M(\g),C_r(\g)} =\sum_{\alpha\in \mathrm{Irr}(\g)}d_{\alpha}\mathrm{tr}(\widehat{\mu}(\alpha)Q_{\alpha}\widehat{f}(\alpha)^*).\]

\end{lemma}

\begin{proof}
For $f=\displaystyle \sum_{\alpha\in \mathrm{Irr}(\g)}\sum_{i,j=1}^{n_{\alpha}}d_{\alpha}(\widehat{f}(\alpha)Q_{\alpha})_{i,j}u^{\alpha}_{j,i}\in \mathrm{Pol}(\g)$,  
\[f^*=\displaystyle \sum_{\alpha\in \mathrm{Irr}(\g)}\sum_{i,j=1}^{n_{\alpha}}d_{\alpha}\overline{(\widehat{f}(\alpha)Q_{\alpha})_{i,j}}(u^{\alpha}_{j,i})^*\]
and
\begin{align*}
\langle\mu,f^*\rangle_{M(\g),C_r(\g)}&=\sum_{\alpha\in \mathrm{Irr}(\g)}\sum_{i,j=1}^{n_{\alpha}}d_{\alpha}\overline{(\widehat{f}(\alpha)Q_{\alpha})_{i,j}}\langle \mu,(u^{\alpha}_{j,i})^*\rangle_{M(\g),C_r(\g)}\\
&=\sum_{\alpha\in \mathrm{Irr}(\g)}\sum_{i,j=1}^{n_{\alpha}}d_{\alpha}\overline{(\widehat{f}(\alpha)Q_{\alpha})_{i,j}}[\widehat{\mu}(\alpha)]_{i,j}\\
&=\sum_{\alpha\in \mathrm{Irr}(\g)}\sum_{i,j=1}^{n_{\alpha}}d_{\alpha}\overline{(\widehat{f}(\alpha)Q_{\alpha})_{i,j}[\widehat{\mu}(\alpha)^*]_{j,i}}\\
&=\sum_{\alpha\in \mathrm{Irr}(\g)}d_{\alpha}\overline{\mathrm{tr}(\widehat{f}(\alpha)Q_{\alpha}\widehat{\mu}(\alpha)^*)}\\
&=\sum_{\alpha\in \mathrm{Irr}(\g)}d_{\alpha}\mathrm{tr}(\widehat{\mu}(\alpha)Q_{\alpha}\widehat{f}(\alpha)^*).
\end{align*}

\end{proof}

\begin{proposition}\label{lem2}
Fix a family of matrices $(X_{\alpha})_{\alpha\in \mathrm{Irr}(\g)}$. Then
\[\mu_U\sim\sum_{\alpha\in \mathrm{Irr}(\g)}d_{\alpha}\mathrm{tr}( U_{\alpha}X_{\alpha}Q_{\alpha}u^{\alpha}) \in M(\g)\]
for all $U=(U_{\alpha})_{\alpha\in \mathrm{Irr}(\g)}\in \displaystyle \mathcal{U}$
if and only if
\[\mu_B\sim \sum_{\alpha\in \mathrm{Irr}(\g)}d_{\alpha}\mathrm{tr}( B_{\alpha}X_{\alpha}Q_{\alpha}u^{\alpha}) \in M(\mathbb{G})\]
for all $B=(B_{\alpha})_{\alpha\in \mathrm{Irr}(\g)}\in \ell^{\infty}(\widehat{\g})$.

In this case, the map 
\[\Phi:\ell^{\infty}(\widehat{\g})\rightarrow M(\g), B=(B_{\alpha})_{\alpha\in \mathrm{Irr}(\g)}\mapsto \mu_{B}\sim \sum_{\alpha\in \mathrm{Irr}(\mathbb{G})}d_{\alpha}\mathrm{tr}( B_{\alpha}X_{\alpha}Q_{\alpha}u^{\alpha}),\]
is automatically a bounded map by the closed graph theorem.

\end{proposition}

\begin{proof}
One direction is trivial. On the other direction, pick $B=(B_{\alpha})_{\alpha\in \mathrm{Irr}(\g)}\in \mathrm{Ball}(\ell^{\infty}(\widehat{\g}))$ and take an arbitrary $f\in \mathrm{Pol}(\mathbb{G})$. Then
\[\langle \mu_B,f^*\rangle_{M(\g),C_r(\g)}=\sum_{\alpha\in \mathrm{Irr}(\g)}d_{\alpha}\mathrm{tr}(B_{\alpha}X_{\alpha}Q_{\alpha}\widehat{f}(\alpha)^*).\]

For each $\alpha\in \mathrm{Irr}(\g)$, there exist unitaries $v_{\alpha}^j$ with $j=1,\cdots,4$ such that \[X_{\alpha}=\displaystyle \frac{\sum_{j=1}^4 v_{\alpha}^j}{2}\]
by considering the following decomposition
\[ X_{\alpha}=\frac{X_{\alpha}^*+X_{\alpha}}{2}+i \frac{X_{\alpha}-X_{\alpha}^*}{2i}=:h_1+i h_2\]
and
\[h_j= (\frac{h_j +i\sqrt{1-\left |h_j\right |^2}}{2}+\frac{ h_j -i\sqrt{1-\left |h_j\right |^2}}{2})\mathrm{~for~}j=1,2.\]

We define $V_j=(v_{\alpha}^j)_{\alpha\in \mathrm{Irr}(\g)}\in \displaystyle \mathcal{U}$ for $j=1,\cdots, 4$ and then
\[\langle \mu_B,f^*\rangle_{M(\g),C_r(\g)}=\frac{1}{2}\sum_{j=1}^4\sum_{\alpha\in \mathrm{Irr}(\g)}d_{\alpha}\mathrm{tr}(v^j_{\alpha}A_{\alpha}Q_{\alpha}\widehat{f}(\alpha)^*)=\frac{1}{2}\sum_{j=1}^4\langle \mu_{V_j},f^*\rangle_{M(\g),C_r(\g)}.\]

Since $\mathrm{Pol}(\mathbb{G})$ is dense in $C_r(\mathbb{G})$, we get the conclusion
\[\mu_B=\frac{\sum_{j=1}^4 \mu_{V_j}}{2}\in M(\g).\]
\end{proof}

\subsection{Vector valued probabilistic methods}

In this subsection, we gather some probabilistic techniques, namely non-commutative Khintchine inequality and cotype 2 condition, to be used in this study.

First of all, let us write down the non-commutative Khintchine inequality in the context of the present  compact quantum group setting. For more details, see \cite{Pi12}.

For a sequence $(x_j)_{j=1}^n\subseteq L^1(\g)$, we set
\[\vvvert (x_j)_{j=1}^n \vvvert_1 =\sup \left \{ \sum \langle x_j,y_j\rangle_{L^1(\g),L^{\infty}(\g)}:y_j\in L^{\infty}(\g),~\max \left \{ \left \|(y_j)_{j=1}^n\right\|_C,\left \|(y_j)_{j=1}^n \right\|_R \right\}\leq 1 \right\}\]
where $\left \|(y_j)_{j=1}^n\right\|_C=\displaystyle \left \|(\sum_{j=1}^n y_j^*y_j )^{\frac{1}{2}}\right\|_{L^{\infty}(\g)}$ and $\left \| (y_j)_{j=1}^n\right\|_R=\displaystyle \left \|(\sum_{j=1}^n y_jy_j^*)^{\frac{1}{2}}\right\|_{L^{\infty}(\g)}$.

Let $(g_j)_{j=1}^n$ be an independent and identically distributed(in short, i.i.d.) sequece of real valued Gaussian random variables with mean zero and variance 1. Then a part of the Khintchine inequality for non-commutative $L^1$-spaces is as follows:

\begin{theorem}[Theorem 9.1, \cite{Pi12}]\label{thm-Khintchine}
If $\g$ is of Kac type, there exists a universal constant $C>0$ such that
\[\frac{1}{C}\vvvert (x_j)_{j=1}^n \vvvert_1 \leq \int \left \|\sum_{j=1}^n g_j(w) x_j \right \|d\mathbb{P}(w) \leq \vvvert (x_j)_{j=1}^n \vvvert_1\]
for any finite set $(x_j)_{j=1}^n\subseteq L^1(\g)$.
\end{theorem}

Also, we will make use of a cotype study on non-commutative $L^1$-spaces. In general, we say that a Banach space is {\it of cotype 2} if there is a constant $K>0$ such that for all finite sets $(x_j)_{j=1}^n \subseteq B$ we have 
\[\int \left \| \sum_{j=1}^n \epsilon_j(w) x_j  \right \|_B d\mathbb{P}(w)\geq K (\sum_{j=1}^n \left \|x_j\right \|^2)^{\frac{1}{2}},\]
where $(\epsilon_j)_j$ is a family of i.i.d. Rademacher variables. It is known that every predual of von Neumann algebra is of cotype 2 \cite{To74} and it is possible to replace the Rademacher varialbes $\epsilon_j$ to the Gaussian variables $g_j$ \cite{LeTa13}.

\begin{theorem}\label{thm-cotype}
There exists a universal constant $K>0$ such that
\[\int \left \|\sum_{j=1}^n g_j(w)x_j \right \|_{L^1(\g)}d\mathbb{P}(w) \geq K(\sum_{j=1}^n \left \|x_j\right \|_{L^1(\g)}^2)^{\frac{1}{2}}\] 
for any finite sets $(x_j)_{j=1}^n\subseteq L^1(\g)$.
\end{theorem}

\section{The main results}

\subsection{Affirmative answer on Kac cases}
In this subsection, we will make use of random matrices $G_n=(\displaystyle \frac{1}{\sqrt{n}}g^n_{i,j})_{i,j=1}^n\in M_n$ where $(g_{i,j}^n)_{i,j=1}^n$ is a family of independent real valued Gaussian random variables with mean zero and variance $1$.

\begin{lemma}[Proposition 1.5, Chapter 5, \cite{MaPi81}]\label{lem-gaussian}

There exist universal constants $C_1,C_2>0$, which are independent of $n$, such that
\[C_1\leq \int \left \|G_n (w) \right\|d\mathbb{P}(w) \leq C_2,\]
where $\left \|\cdot \right\|$ is the operator norm of matrices.
\end{lemma}

\begin{theorem}\label{thm-Kac}
Let $\g$ be a compact quantum group of Kac type and suppose that all random fourier series $f_U\displaystyle \sim \sum_{\alpha\in \mathrm{Irr}(\g)}n_{\alpha}\mathrm{tr}(U_{\alpha}\widehat{f}(\alpha)u^{\alpha})$ of $f\in M(\g)$ are in $M(\g)$ for all $U=(U_{\alpha})_{\alpha\in \mathrm{Irr}(\g)}\in \displaystyle \mathcal{U}$. Then we have
\[\sum_{\alpha\in \mathrm{Irr}(\g)}n_{\alpha}\mathrm{tr}(\widehat{f}(\alpha)^*\widehat{f}(\alpha))<\infty.\]
\end{theorem}

\begin{proof}
By Proposition \ref{lem2}, there exists a universal constant $K>0$ such that 
\[K \left \|B\right\|_{\ell^{\infty}(\widehat{\g})}\geq \left \| f_B\sim \sum_{\alpha\in \mathrm{Irr}(\g)}n_{\alpha}\mathrm{tr}(B_{\alpha}\widehat{f}(\alpha)u^{\alpha})\right \|_{M(\g)}\]
for any $B=(B_{\alpha})_{\alpha\in \mathrm{Irr}(\g)}\in \ell^{\infty}(\widehat{\g})$. Now we fix a finite subset $S\subseteq \mathrm{Irr}(\g)$ and naturally takes into account a family of random matrices $(G_{n_{\alpha}})_{\alpha\in S}=\displaystyle ((\frac{1}{\sqrt{n_{\alpha}}}g^{\alpha}_{i,j})_{i,j=1}^{n_{\alpha}})_{\alpha\in S}$ in $\ell^{\infty}(\widehat{\g})$. Then from the estimate above and Theorem \ref{thm-Khintchine}, we have

\begin{align}
\label{Khintchine}K \int \sup_{\alpha\in S}\left \|G_{n_{\alpha}}\right\| d\mathbb{P}(w) &\geq \int \left \|\sum_{\alpha\in S}\sum_{i,j=1}^{n_{\alpha}} \sqrt{n_{\alpha}}g^{\alpha}_{i,j}(w)(\widehat{f}(\alpha)u^{\alpha})_{j,i}  \right\|_{L^1(\g)}d\mathbb{P}(w)\\
\notag&\geq \frac{1}{C} \vvvert (\sqrt{n_{\alpha}}(\widehat{f}(\alpha)u^{\alpha})_{j,i})_{\alpha,i,j} \vvvert_1\\
\notag&=\frac{1}{C}\sup_{(y^{\alpha}_{i,j})_{\alpha,i,j}} \sum \langle x^{\alpha}_{i,j},y^{\alpha}_{i,j}\rangle_{L^1(\g),L^{\infty}(\g)}
\end{align}
where $x^{\alpha}_{i,j}=\sqrt{n_{\alpha}}(\widehat{f}(\alpha)u^{\alpha})_{j,i}$ and the supremum runs over all $(y^{\alpha}_{i,j})_{\alpha,i,j}\subseteq L^{\infty}(\g)$ such that 
\[\max \left \{ \left \| (\sum_{\alpha\in S}\sum_{i,j=1}^{n_{\alpha}} (y^{\alpha}_{i,j})^*y^{\alpha}_{i,j})^{\frac{1}{2}}\right \|_{L^{\infty}(\g)},\left \| (\sum_{\alpha\in S}\sum_{i,j=1}^{n_{\alpha}} y^{\alpha}_{i,j} (y^{\alpha}_{i,j})^* )^{\frac{1}{2}}\right \|_{L^{\infty}(\g)}\right\}\leq 1. \]

For arbitrary $(A(\alpha))_{\alpha\in S}$ such that $\displaystyle \sum_{\alpha\in S} n_{\alpha}\mathrm{tr}(A(\alpha)^*A(\alpha))\leq 1$, we can find that the family $(y^{\alpha}_{i,j})=(\sqrt{n_{\alpha}}[(A(\alpha)u^{\alpha})_{j,i}]^*)$ satisfies the condition above. Indeed,
\begin{align*}
\sum_{\alpha\in S}\sum_{i,j=1}^{n_{\alpha}} (y^{\alpha}_{i,j})^*y^{\alpha}_{i,j} &= \sum_{\alpha\in S}\sum_{i,j=1}^{n_{\alpha}} n_{\alpha}(A(\alpha)u^{\alpha})_{j,i}[(A(\alpha)u^{\alpha})_{j,i}]^*\\
&=\sum_{\alpha\in S}\sum_{i,j=1}^{n_{\alpha}} n_{\alpha}(A(\alpha)u^{\alpha})_{j,i}[(u^{\alpha})^*A(\alpha)^*]_{i,j}\\
&=\sum_{\alpha\in S}n_{\alpha}\mathrm{tr}(A(\alpha)u^{\alpha}(u^{\alpha})^*A(\alpha)^*)\\
&=\sum_{\alpha\in S} n_{\alpha}\mathrm{tr}(A(\alpha) A(\alpha)^*)1_A\\
&=\sum_{\alpha\in S} n_{\alpha}\mathrm{tr}(A(\alpha)^* A(\alpha))1_A
\end{align*}
and
\begin{align*}
\sum_{\alpha\in S}\sum_{i,j=1}^{n_{\alpha}} y^{\alpha}_{i,j}(y^{\alpha})_{i,j}^* &= \sum_{\alpha\in S}\sum_{i,j=1}^{n_{\alpha}} n_{\alpha}[(A(\alpha)u^{\alpha})_{j,i}]^*(A(\alpha)u^{\alpha})_{j,i}\\
&=\sum_{\alpha\in S}\sum_{i,j=1}^{n_{\alpha}} n_{\alpha}(\overline{A(\alpha)}\overline{u^{\alpha}})_{j,i}[(u^{\alpha})^t A(\alpha)^t]_{i,j}\\
&=\sum_{\alpha\in S}n_{\alpha}\mathrm{tr}(\overline{A(\alpha)}\overline{u^{\alpha}}(u^{\alpha})^tA(\alpha)^t)\\
&=\sum_{\alpha\in S} n_{\alpha}\mathrm{tr}(\overline{A(\alpha)} A(\alpha)^t)1_A\\
&=\sum_{\alpha\in S} n_{\alpha}\mathrm{tr}(A(\alpha)^* A(\alpha))1_A ,
\end{align*}
so that we have 
\begin{align*}
&\max\left\{ \left \| (\sum_{\alpha\in S}\sum_{i,j=1}^{n_{\alpha}} (y^{\alpha}_{i,j})^*y^{\alpha}_{i,j})^{\frac{1}{2}}\right \|_{L^{\infty}(\g)},\left \| (\sum_{\alpha\in S}\sum_{i,j=1}^{n_{\alpha}} y^{\alpha}_{i,j} (y^{\alpha}_{i,j})^* )^{\frac{1}{2}}\right \|_{L^{\infty}(\g)}\right \}\\
&\leq (\sum_{\alpha\in S}n_{\alpha}\mathrm{tr}(A(\alpha)^*A(\alpha)))^{\frac{1}{2}}\leq 1
\end{align*}

Finally, by Lemma \ref{lem-gaussian}, we have

\begin{align*}
KC_2&\geq \frac{1}{C} \sup_{(A(\alpha))_{\alpha\in S}}\sum_{\alpha\in S}\sum_{i,j=1}^{n_{\alpha}}\langle \sqrt{n_{\alpha}}(\widehat{f}(\alpha)u^{\alpha})_{j,i}, \sqrt{n_{\alpha}}[(A(\alpha)u^{\alpha})_{j,i}]^*\rangle_{L^1(\g),L^{\infty}(\g)}\\
&=\frac{1}{C}\sup_{(A(\alpha))_{\alpha\in S}}\sum_{\alpha\in S}\sum_{i,j=1}^{n_{\alpha}} n_{\alpha}\sum_{k,l=1}^{n_{\alpha}} \widehat{f}(\alpha)_{j,k}\overline{A(\alpha)_{j,l}} \langle u^{\alpha}_{k,i}, (u^{\alpha}_{l,i})^*\rangle_{L^1(\g),L^{\infty}(\g)}\\
&=\frac{1}{C}\sup_{(A(\alpha))_{\alpha\in S}}\sum_{\alpha\in S}\sum_{i,j}^{n_{\alpha}}n_{\alpha}\sum_{k,l=1}^{n_{\alpha}}\frac{\delta_{k,l}}{n_{\alpha}}\widehat{f}(\alpha)_{j,k}\overline{A(\alpha)_{j,l}}\\
&=\frac{1}{C}\sup_{(A(\alpha))_{\alpha\in S}}\sum_{\alpha\in S} n_{\alpha}\mathrm{tr}(\widehat{f}(\alpha)A(\alpha)^*)\\
&=\frac{1}{C} (\sum_{\alpha\in S}n_{\alpha}\mathrm{tr}(\widehat{f}(\alpha)^*\widehat{f}(\alpha)))^{\frac{1}{2}}
\end{align*}
since the supremum runs over all $(A(\alpha))_{\alpha\in S}$ such that $\displaystyle \sum_{\alpha\in S}n_{\alpha}\mathrm{tr}(A(\alpha)^*A(\alpha))\leq 1$. Now we reach the conclusion since the finite set $S$ is chosen to be arbitrary.

\end{proof}

\begin{remark}\label{rmk-Khintchine}
In the category of compact groups, it seems that some experts already have their own way to understand Theorem \ref{thm1} through the Khintchine inequality. Indeed, the arguments below (\ref{Khintchine}) are simplified as follows. 

\begin{align*}
&\int \left \|\sum_{\pi\in S}\sum_{i,j=1}^{n_{\pi}}\sqrt{n_{\pi}}g^{\pi}_{i,j}(w)(\widehat{f}(\pi)\pi(x))_{j,i} \right \|_{L^1(G)}d\mathbb{P}(w)\\
&=\int_G \int \left |\sum_{\pi\in S}\sum_{i,j=1}^{n_{\pi}}\sqrt{n_{\pi}}g^{\pi}_{i,j}(w)(\widehat{f}(\pi)\pi(x))_{j,i}  \right |  d\mathbb{P}(w) dx\\
&\gtrsim \int_G (\sum_{\pi\in S}\sum_{i,j=1}^{n_{\pi}}n_{\pi} \left |(\widehat{f}(\pi)\pi(x))_{j,i}\right |^2)^{\frac{1}{2}} dx\\
&= \int_G (\sum_{\pi\in S} n_{\pi}\mathrm{tr}(\widehat{f}(\pi)\pi(x)\pi(x)^*\widehat{f}(\pi)^*))^{\frac{1}{2}} dx\\
&=(\sum_{\pi\in S} n_{\pi} \mathrm{tr}(\widehat{f}(\pi)^*\widehat{f}(\pi)))^{\frac{1}{2}}.
\end{align*}

\end{remark}

\begin{remark}
Let $\g$ be a compact quantum group of Kac type. If we suppose that $\displaystyle \sum_{\alpha\in \mathrm{Irr}(\g)}n_{\alpha}\mathrm{tr}(X_{\alpha}^*X_{\alpha})<\infty$, then the multiplier 
\[\Phi:L^p(\g)\rightarrow \ell^1(\widehat{\g}), f\mapsto (X_{\alpha}\widehat{f}(\alpha))_{\alpha\in \mathrm{Irr}(\g)},\]
is bounded for all $2\leq p\leq \infty$ since $L^p(\g)\subseteq L^2(\g)$ for such $p$. The converse is also true thanks to Theorem \ref{thm-Kac}. Indeed, boundedness of $\Phi$ implies that the multiplier
\[\Psi:\ell^{\infty}(\widehat{\g})\rightarrow M(\g),(A_{\alpha})_{\alpha\in \mathrm{Irr}(\g)}\mapsto \sum_{\alpha\in \mathrm{Irr}(\g)}d_{\alpha}\mathrm{tr}(A_{\alpha}X_{\alpha}^*u^{\alpha})\]
is bounded. Then, by Proposition \ref{lem2} and Theorem \ref{thm-Kac}, we can conclude that the space of such multipliers from $L^p(\g)$ into $\ell^1(\widehat{\g})$ can be identified with $L^2(\g)$ independently of the choice of $2\leq p\leq \infty$.
\end{remark}

\subsection{A partial answer on general cases}

In spite of the complete answer on the Kac case, the proof of Theorem \ref{thm-Kac} could not be applied for the general case for now. However, using the cotype 2 property of $L^1(\g)$, we are able to achieve a convincing conclusion for non-Kac cases.

\begin{lemma}\label{lem3}
\begin{enumerate}
\item For any $\alpha\in \mathrm{Irr}(\g)$ and any $1\leq i,j\leq n_{\alpha}$, we have
\[\left \|\sum_{k=1}^{n_{\alpha}}A_{ik}(u^{\alpha}_{k,j})^*\right \|_{L^{\infty}(\g)}\leq (Q_{\alpha}^{\frac{1}{2}})_{j,j}(\sum_{k=1}^{n_{\alpha}}|(A_{\alpha}Q_{\alpha}^{-\frac{1}{2}})_{i,k}|^2)^{\frac{1}{2}}.\]

\item For any $\alpha\in \mathrm{Irr}(\g)$ and any $1\leq i,j\leq n_{\alpha}$, we have
\[\left \|\sum_{k=1}^{n_{\alpha}}(BQ_{\alpha})_{i,k}u^{\alpha}_{k,j}\right \|_{L^1(\g)}\geq \frac{(Q_{\alpha}^{-\frac{1}{2}})_{j,j}}{d_{\alpha}}(\sum_{k=1}^{n_{\alpha}}|(BQ_{\alpha}^{\frac{1}{2}})_{i,k}|^2)^{\frac{1}{2}}.\]
\end{enumerate}
\end{lemma}
\begin{proof}
\begin{enumerate}
\item Since $Q_{\alpha}^{\frac{1}{2}}\overline{u^{\alpha}}Q_{\alpha}^{-\frac{1}{2}}$ is a unitary representation, we have
\begin{align*}
&(Q_{\alpha}^{-\frac{1}{2}})_{j,j}\left \|\sum_{k=1}^{n_{\alpha}}A_{i,k}(u^{\alpha}_{k,j})^*\right \|_{L^{\infty}(\g)}\\
=~&\left \|\sum_{k=1}^{n_{\alpha}}(AQ_{\alpha}^{-\frac{1}{2}})_{i,k}[Q_{\alpha}^{\frac{1}{2}}\overline{u}Q_{\alpha}^{-\frac{1}{2}}]_{k,j}\right \|_{L^{\infty}(\g)}\\
\leq ~& (\sum_{k=1}^{n_{\alpha}}|(AQ_{\alpha}^{-\frac{1}{2}})_{i,k}|^2)^{\frac{1}{2}}(\sum_{k=1}^{n_{\alpha}}[Q_{\alpha}^{\frac{1}{2}}\overline{u}Q_{\alpha}^{-\frac{1}{2}}]_{k,j}^*[Q_{\alpha}^{\frac{1}{2}}\overline{u}Q_{\alpha}^{-\frac{1}{2}}]_{k,j})^{\frac{1}{2}}\\
=&(\sum_{k=1}^{n_{\alpha}}|(AQ_{\alpha}^{-\frac{1}{2}})_{i,k}|^2)^{\frac{1}{2}}.
\end{align*}

\item Define $w_{i,k}=(AQ_{\alpha}^{-\frac{1}{2}})_{i,k}$ for all $1\leq i,k\leq n_{\alpha}$ for convenience. Then, by $(1)$, we can see that
\begin{align*}
&\left \|\sum_{k=1}^{n_{\alpha}}(BQ_{\alpha})_{i,k}u^{\alpha}_{k,j}\right \|_{L^1(\g)}\\
&\geq \sup_{(\sum_{k=1}^{n_{\alpha}}|w_{i,k}|^2)^{\frac{1}{2}}\leq 1}\langle \sum_{k=1}^{n_{\alpha}}(BQ_{\alpha})_{i,k}u^{\alpha}_{k,j}, (Q_{\alpha}^{-\frac{1}{2}})_{j,j}\sum_{k=1}^{n_{\alpha}}A_{i,k}(u^{\alpha}_{k,j})^* \rangle_{L^1(\g),L^{\infty}(\g)}\\
&=(Q_{\alpha}^{-\frac{1}{2}})_{j,j}\sup_{(\sum_{k=1}^{n_{\alpha}}|w_{i,k}|^2)^{\frac{1}{2}}\leq 1}\sum_{k=1}^{n_{\alpha}}(BQ_{\alpha})_{i,k}w_{i,k}(Q_{\alpha}^{\frac{1}{2}})_{k,k}h((u^{\alpha}_{k,j})^*u^{\alpha}_{k,j})\\
&=(Q_{\alpha}^{-\frac{1}{2}})_{j,j}\sup_{(\sum_{k=1}^{n_{\alpha}}|w_{i,k}|^2)^{\frac{1}{2}}\leq 1}\sum_{k=1}^{n_{\alpha}}\frac{(BQ_{\alpha}^{\frac{1}{2}})_{i,k}}{d_{\alpha}}w_{i,k}\\
&=(Q_{\alpha}^{-\frac{1}{2}})_{j,j}(\sum_{k=1}^{n_{\alpha}}\frac{|(BQ_{\alpha}^{\frac{1}{2}})_{i,k}|^2}{d_{\alpha}^2})^{\frac{1}{2}}=\frac{(Q_{\alpha}^{-\frac{1}{2}})_{j,j}}{d_{\alpha}}(\sum_{k=1}^{n_{\alpha}} |(BQ_{\alpha}^{\frac{1}{2}})_{i,k}|^2)^{\frac{1}{2}}
\end{align*}

\end{enumerate}
\end{proof}

\begin{theorem}\label{thm-nonKac}
Suppose that all random Fourier series $f_U\sim \displaystyle \sum_{\alpha\in \mathrm{Irr}(\g)}d_{\alpha}\mathrm{tr}(U_{\alpha}\widehat{f}(\alpha)Q_{\alpha}u^{\alpha})$ are in $M(\g)$ for all $U=\displaystyle (U_{\alpha})_{\alpha\in \mathrm{Irr}(\g)}\in \mathcal{U}$. Then we have 
\[\sum_{\alpha\in \mathrm{Irr}(\g)}\frac{d_{\alpha}}{n_{\alpha}}\mathrm{tr}(Q_{\alpha}\widehat{f}(\alpha)^*\widehat{f}(\alpha))<\infty.\]
\end{theorem}

\begin{proof}
As in the proof of Theorem \ref{thm-Kac}, there exists a universal constant $K>0$, which is independent of the choice of finite subset $S\subseteq \mathrm{Irr}(\g)$, such that
\begin{align*}
K&\geq \int  \left \|\sum_{\alpha\in S}\sum_{i,j=1}^{n_{\alpha}}\frac{d_{\alpha}}{\sqrt{n_{\alpha}}}g^{\alpha}_{j,i}(w)\sum_{k=1}^{n_{\alpha}}(\widehat{f}(\alpha)Q_{\alpha})_{i,k}u^{\alpha}_{k,j}\right \|_{L^1(\g)}d\mathbb{P}(w)
\end{align*}
where $(g^{\alpha}_{j,i}(w))_{\alpha\in \mathrm{Irr}(\g),1\leq i,j\leq n_{\alpha}}$ is a i.i.d. family of Gaussian variables with mean zero and variance $1$.

Let $v^{\alpha}_i$ be the $i$-th row vector of the matrix $\widehat{f}(\alpha)Q_{\alpha}^{\frac{1}{2}}$. Then, by Lemma \ref{lem3} and Theorem \ref{thm-cotype}, we have that
\begin{align*}
K&\geq K' (\sum_{\alpha\in S}\sum_{i,j=1}^{n_{\alpha}}\frac{d_{\alpha}^2}{n_{\alpha}}\left \| \sum_{k=1}^{n_{\alpha}}(\widehat{f}(\alpha)Q_{\alpha})_{i,k}u^{\alpha}_{k,j} \right\|_{L^1(\g)}^2)^{\frac{1}{2}}\\
&\geq K' (\sum_{\alpha\in S}\sum_{i,j=1}^{n_{\alpha}}\frac{(Q_{\alpha}^{-1})_{j,j}\left \| v^{\alpha}_i \right\| ^2}{n_{\alpha}})^{\frac{1}{2}}\\
&=(\sum_{\alpha\in S}\frac{d_{\alpha}}{n_{\alpha}}\mathrm{tr}(Q_{\alpha}\widehat{f}(\alpha)^*\widehat{f}(\alpha)))^{\frac{1}{2}}.
\end{align*}

Since $S\subseteq \mathrm{Irr}(\g)$ is arbitrary, we can conclude that 
\[ \sum_{\alpha\in \mathrm{Irr}(\g)}\frac{d_{\alpha}}{n_{\alpha}}\mathrm{tr}(Q_{\alpha}\widehat{f}(\alpha)^*\widehat{f}(\alpha))<\infty.\]

\end{proof}

\begin{remark}\label{rmk1}
Although Theorem \ref{thm-nonKac} does not give the complete conclusion for non-Kac cases, we can expect that the quantity $\displaystyle \frac{d_{\alpha}}{n_{\alpha}}$ gets very close to the desired quantity of $d_{\alpha}$. At least in certain cases, $n_{\alpha}$ is significantly smaller than the quantum dimension $d_{\alpha}$. Indeed, if the function $\alpha\mapsto n_{\alpha}$ has subexponential growth and if $\g$ is of non-Kac type, then the function $\alpha\mapsto d_{\alpha}$ is of exponential growth \cite{DPR16}. In particular, the Drinfeld-Jimbo deformations $G_q$ with $0<q<1$ are important non-Kac type compact quantum groups. In this case, the function $\alpha\mapsto n_{\alpha}$ is of subexponential growth.
\end{remark}

Let us clarify what Remark \ref{rmk1} implies for concrete example $\g=SU_q(2)$.

\begin{corollary}
Let $\g=SU_q(2)$ with $0<q<1$ and suppose that all random fourier series $f_U\sim \displaystyle \sum_{\alpha\in \mathrm{Irr}(\g)}d_{\alpha}\mathrm{tr}(U_{\alpha}\widehat{f}(\alpha)Q_{\alpha}u^{\alpha})$ are in $M(\g)$ for all $U\displaystyle =(U_{\alpha})_{\alpha\in \mathrm{Irr}(\g)}\in \mathcal{U}$. Then we have
\[\sum_{\alpha\in \mathrm{Irr}(\g)}d_{\alpha}^{1-\epsilon}\mathrm{tr}(Q_{\alpha}\widehat{f}(\alpha)^*\widehat{f}(\alpha))<\infty~for~each~\epsilon>0.\]
\end{corollary}
\begin{proof}
It is known that $\mathrm{Irr}(\g)$ is identified with $\left \{0\right\}\cup \n$ and $n_k=k+1$ for all $k\in \mathrm{Irr}(\g)=\left \{0\right\}\cup \n$. Also,  $\displaystyle d_k=q^{-k}+q^{-k+2}+\cdots+q^k \geq q^{-k}$ for all $k\geq 0$. Therefore, for each $\epsilon>0$,
\begin{align*}
\sum_{k\geq 0}d_k^{1-\epsilon}\mathrm{tr}(Q_k\widehat{f}(k)^*\widehat{f}(k))&\leq \sum_{k\geq 0} \frac{n_k}{d_k^{\epsilon}}\frac{d_k}{n_k}\mathrm{tr}(Q_k\widehat{f}(k)^*\widehat{f}(k))\\
&\leq \sum_{k\geq 0}\frac{k+1}{q^{-\epsilon k}}\frac{d_k}{n_k}\mathrm{tr}(Q_k\widehat{f}(k)^*\widehat{f}(k))\\
&\leq \frac{1}{(1-q^{\epsilon})^2}\sum_{k\geq 0}\frac{d_k}{n_k}\mathrm{tr}(Q_k\widehat{f}(k)^*\widehat{f}(k))<\infty.
\end{align*}
\end{proof}

\subsection{A remark on central forms}

While Theorem \ref{thm-nonKac} clearly contains insufficient points to reach the conclusion, we conjecture that the ultimate conclusion might turn out to be positive. The reason is that if the starting measure $f\in M(\g)$ has a specific form, the so-called central form, it can offset the deficit considerably.

Throughout this subsection, we will consider a sub von Neumann algebra $N$ generated by $\left \{ \chi_{\alpha}\right\}_{\alpha\in \mathrm{Irr}(\g)}$ in $L^{\infty}(\g)$. Then the restriction of Haar state to $N$ is a finite tracial state since 
\[h(\chi_{\alpha}\chi_{\beta})=h(\chi_{\overline{\alpha}}^*\chi_{\beta})=\delta_{\overline{\alpha},\beta}=\delta_{\overline{\beta},\alpha}=h(\chi_{\overline{\beta}}^*,\chi_{\alpha})=h(\chi_{\beta}\chi_{\alpha})~\mathrm{for~all~}\alpha,\beta\in \mathrm{Irr}(\g).\]

\begin{lemma}\label{lem4}
Let $\mathbb{G}$ be a compact quantum group satisfying $\displaystyle \inf_{\alpha\in \mathrm{Irr}(\g)}h(|\chi_{\alpha}|)>0$. Then there exists a universal constant $K$ satisfying the following: for any sequence $(a_{\alpha})_{\alpha\in \mathrm{Irr}(\g)}$ with $\displaystyle \sum_{\alpha\in \mathrm{Irr}(\mathbb{G})}|a_{\alpha}|^2\leq 1$, there exists $\displaystyle (T_i)_i\subseteq \mathrm{Pol}(\g)\cap K\cdot \mathrm{Ball}(L^{\infty}(\g))$ such that $\displaystyle \lim_i \left |h(T_i \chi_{\alpha}^*)\right |\geq |a_{\alpha}|$ for all $\alpha\in \mathrm{Irr}(\g)$.
\end{lemma}

\begin{proof}

By [Theorem 5, \cite{Lu-Pi97}], there exists a universal constant $K>0$ and for any such sequence $(a_{\alpha})_{\alpha\in \mathrm{Irr}(\g)}$, there exists $T\in K\cdot \mathrm{Ball}(N)$ such that $\displaystyle T\sim \sum_{\alpha\in \mathrm{Irr}(\g)}c_{\alpha}\chi_{\alpha}$ with $|c_{\alpha}|\geq |a_{\alpha}|$ for all $\alpha\in \mathrm{Irr}(\g)$.

Lastly, since $\mathcal{A}=\mathrm{Pol}(\g)\cap N$ is a unital $*$-algebra and dense in $N$ under the weak $*$-topology, Ball($\mathcal{A})$ is dense in Ball$ (N)$ by the Kaplansky's density theorem. Thus, we obtain a net $(T_i)_i \subseteq \mathrm{Pol}(\g)\cap K\cdot \mathrm{Ball}(L^{\infty}(\g))$ satisfying
\[\left | c_{\alpha}\right |=h(T\chi_{\alpha}^*)=\lim_i \left |h(T_i \chi_{\alpha}^*)\right |\geq \left |a_{\alpha}\right |\]
for all $\alpha\in \mathrm{Irr}(\g)$.
\end{proof}

\begin{theorem}\label{thm-central}
Let $\g$ be a compact quantum group such that $\displaystyle \inf_{\alpha\in \mathrm{Irr}(\g)}h(|\chi_{\alpha}|)>0$ aud suppose that all random Fourier series $f_U\sim \displaystyle \sum_{\alpha\in \mathrm{Irr}(\g)}d_{\alpha}\mathrm{tr}(U_{\alpha}\widehat{f}(\alpha)Q_{\alpha}u^{\alpha})$ are in $M(\g)$ for all $U=\displaystyle (U_{\alpha})_{\alpha\in \mathrm{Irr}(\g)}\in \mathcal{U}$. Then we have
\[\sum_{\alpha\in \mathrm{Irr}(\g)}\mathrm{tr}(|\widehat{f}(\alpha)|)^2<\infty.\]

\end{theorem}

\begin{proof}
Let $(a_{\alpha})_{\alpha\in \mathrm{Irr}(\g)}$ be a sequence with $\displaystyle \sum_{\alpha\in \mathrm{Irr}(\g)}|a_{\alpha}|^2\leq 1$. Then we have a family $\left \{T_i\right\}_i \subseteq \mathrm{Pol}(\g)\cap K\cdot \mathrm{Ball}(L^{\infty}(\g))$ coming from Lemma \ref{lem4}. Let us write $T_i$ as $\displaystyle \sum_{\alpha\in \mathrm{Irr}(\g)}c^i_{\alpha}\chi_{\alpha}$. Then $\widehat{T_i}(\alpha)Q_{\alpha}=\displaystyle \frac{c^i_{\alpha}}{d_{\alpha}}\mathrm{Id}_{n_{\alpha}}$ for all $\alpha\in \mathrm{Irr}(\g)$. Therefore, using Proposition \ref{lem2}, there exists a universal constant $K'>0$ such that 
\begin{align*}
\infty> K'&\geq \sup_U \sup_i \langle f_U,T_i^* \rangle_{L^1(\g),L^{\infty}(\g)}\\
&=\sup_U \sup_i \sum_{\alpha\in \mathrm{Irr}(\g)}d_{\alpha}\mathrm{tr}(U_{\alpha}\widehat{f}(\alpha)Q_{\alpha}\widehat{T_i}(\alpha)^*)\\
&= \sup_U \sup_i \sum_{\alpha\in \mathrm{Irr}(\g)} \overline{c_{\alpha}^i} \mathrm{tr}(U_{\alpha}\widehat{f}(\alpha))\\
&=\sup_i \sum_{\alpha\in \mathrm{Irr}(\g)}\left |c_{\alpha}^i\right |\mathrm{tr}(\left | \widehat{f}(\alpha)\right |)\\
&\geq  \sum_{\alpha\in \mathrm{Irr}(\g)}\left |c_{\alpha}\right |\mathrm{tr}(\left | \widehat{f}(\alpha)\right |)\\
&\geq \sum_{\alpha\in \mathrm{Irr}(\g)}\left |a_{\alpha}\right |\mathrm{tr}(\left | \widehat{f}(\alpha)\right |).
\end{align*}
The conclusion can be reached as stated above since $(a_{\alpha})$ is arbitrary.
\end{proof}

\begin{corollary}\label{cor-central}
Let $\g$ be a compact quantum group such that $\displaystyle \inf_{\alpha\in \mathrm{Irr}(\g)}h(|\chi_{\alpha}|)>0$ and fix a central element $\displaystyle f\sim \sum_{\alpha\in \mathrm{Irr}(\g)}c_{\alpha}\chi_{\alpha}\in M(\g)$. If all random Fourier series $f_U\sim \displaystyle \sum_{\alpha\in \mathrm{Irr}(\g)}d_{\alpha}\mathrm{tr}(U_{\alpha}\widehat{f}(\alpha)Q_{\alpha}u^{\alpha})$ are in $M(\g)$ for all $U=\displaystyle (U_{\alpha})_{\alpha\in \mathrm{Irr}(\g)}\in \mathcal{U}$, then we have
\[\sum_{\alpha\in \mathrm{Irr}(\g)}d_{\alpha}\mathrm{tr}(Q_{\alpha}\widehat{f}(\alpha)^*\widehat{f}(\alpha))=\sum_{\alpha\in \mathrm{Irr}(\g)} \left  |c_{\alpha}\right |^2<\infty.\]
\end{corollary}

\begin{proposition}
The quantum groups listed in the following list are known to satisfy the condition of Corollary \ref{cor-central}:
\begin{itemize}
\item Free orthogonal quantum groups $O_F^+$
\item Quantum automorphism group $\g_{aut}(B,\psi)$ with a $\delta$-form $\psi$
\item Drinfeld-Jimbo $q$ deformations $G_q$ with $0<q<1$
\end{itemize}
\end{proposition}

\begin{proof}
Let $\g_1$ and $\g_2$ be compact quantum groups. If there exists a bijective map $\Phi:\mathrm{Irr}(\g_1)\rightarrow \mathrm{Irr}(\g_2)$ such that
\[\Phi(\pi_1\otimes \pi_2)=\Phi(\pi_1)\otimes \Phi(\pi_2)~\mathrm{and~}\Phi(\oplus_{i=1}^n \pi_i)=\oplus_{i=1}^n \Phi(\pi_i)\]
for all $\pi_i\in \mathrm{Irr}(\g_1)$ and $n\geq 1$, then the map $\Phi$ extends to a bijective $*$-homomorphism $\Phi: N_1\rightarrow N_2$ where $N_j$ is the weak $*$-closure of $\mathrm{span}\left \{\chi^j_{\alpha}\right\}_{\alpha\in \mathrm{Irr}(\g_j)}$ in $L^{\infty}(\g_j)$ for each $j=1,2$. Also, recall that the Haar states $h_j$ are tracial on $N_j$ respectiely. Moreover, $\Phi$ is trace-preserving by repeating the proofs of [Proposition 6.7, \cite{Wa17}] and [Lemma 4.7, \cite{Yo17}]. Therefore, we can conclude that
\[h_1(\left |\chi^1_{\alpha}\right |)=h_2(\left |\chi^2_{\Phi(\alpha)}\right |)~\mathrm{for~all~}\alpha\in \mathrm{Irr}(\g).\]

On the other hand, in \cite{Pr75}, it is shown that
\[\inf_{\pi\in \mathrm{Irr}(G)}\left \|\chi_{\pi}\right \|_{L^1(G)}>0\]
for any compact connected Lie group. 

The existence of the bijective map $\Phi:\mathrm{Irr}(\g)\rightarrow \mathrm{Irr}(G)$ is known for $(\g,G)=(O_N^+,SU(2))$, $(\g_{aut}(B,\psi),SO(3))$ and $(G_q,G)$. Refer to \cite{Br13}, \cite{NeTu13} and \cite{Ti08} for details.

\end{proof}

\section{Application to complete representability problem}\label{sec:application}

In the category of Banach algebras, $C^*$-algebras are in a special position. One of typical examples that are not $C^*$-algebras is the convolution algebra $L^1(G)$ of a locally compact group. Moreover, it turned out that $L^1(G)$ is isomorphic to a closed subalgebra of $B(H)$ for some Hilbert $H$ as Banach algebras if and only if $G$ is finite. This fact is based on the study of the Arens irregularity for convolution algebras $L^1(G)$ \cite{Yo73}, which shows $L^1(G)$ is Arens regular only if $G$ is finite.

In the framework of locally compact quantum groups, the Fourier algebra $A(G)$ is understood as a dual object of the convolution algebra $L^1(G)$ in view of the generalized Pontrjagin duality. Very recently, it has shown that the Fourier algebra is {\it completely isomorphic} to a closed subalgebra of $B(H)$ for some Hilbert space if and only if $G$ is finite (\cite{LeYo17} for discrete groups and \cite{LeSaSp16} for general cases). Here, complete isomorphism is the natural isomorphism in the category of operator spaces. Some terminologies of operator space theory required in this section will be explained in the below. 

To my knowledge, prior to \cite{LeYo17}, all conclusions in this direction were based on Arens regularity studies. In the quantum group setting, the situation is similar. Thanks to \cite{HuNeRu12} investigating Arens irregularity for convolution algebras $L^1(\g)$ of locally compact quantum groups, we know that the convolution algebra $L^1(\g)$ is completely isomorphic to a closed subalgebra of $B(H)$ if and only if $L^{\infty}(\g)$ is finite dimensional for $\g$ a {\it co-amenable} compact quantum group.

Our contribution in this direction is the following fact, which removes the additional condition of {\it co-amenability} in the category of compact quantum groups.

\begin{theorem}\label{main6}
Let $\g=(A,\Delta)$ be a compact quantum group. Then $L^1(\g)$ is completely representable as an operator algebra, i.e. $L^1(\g)$ is completely isomorphic to a closed subalgebra of $B(H)$ for some Hilbert space $H$ if and only if $A$ is finite dimensional. 
\end{theorem}

An abstarct operator space is a vector space $E$ equipped with a matrical norm structure on $M_n(E)=M_n\otimes E$ for which
\begin{enumerate}
\item $\left \|v_1\oplus v_2\right\|_{M_{m+n}(E)}\leq \max \left \{\left \|v_1\right\|_m,\left \|v_2\right\|_n\right\}$ for all $v_1\in M_m(E)$ and $v_2\in M_n(V)$ and
\item $\left \|avb\right\|_{M_m(E)}\leq \left \|a\right\|_{M_{m,n}} \left \|v\right\|_{M_n(E)} \left \|b\right\|_{M_{n,m}}$ for all $a\in M_{m,n},b\in M_{n,m}$ and $v\in M_n(E)$.
\end{enumerate}

We say that a linear map $L:E\rightarrow F$ between the two operator spaces $E$ and $F$ is completely bounded if $\displaystyle \sup_{n\in \n}\left \|\mathrm{id}_n\otimes L:M_n(E)\rightarrow M_n(F)\right\|<\infty$.

The two operator spaces that are noted in this section are the convolution algebra $L^1(\g)$ and the extended Haagerup tensor product space $L^{\infty}(\g)\otimes_{eh} L^{\infty}(\g)$. Each spaces has a natural operator space structure. Let us present equivalent description of those operator spaces.

First of all, for $F=(f_{i,j})_{i,j=1}^{n} \in M_n(L^1(\g))$, the natural matricial norm structure is given by
\[\left \|F\right \|_{M_n(L^1(\g))}=\sup_{m\in \n}\sup_{X}\left \| (\langle f_{i,j},x_{k,l}\rangle_{L^1(\g),L^{\infty}(\g)})_{i,j=1}^n ~_{k,l=1}^m \right\|_{M_{mn}},\]
where the supremum runs over all $X=(x_{k,l})_{k,l=1}^m\in \mathrm{Ball}(M_m(L^{\infty}(\g)))$ under the identification $M_m(L^{\infty}(\g))\subseteq M_m(B(L^2(\g)))\cong B(L^2(\g)\oplus \cdots \oplus L^2(\g))$.

Secondly, 
\[L^{\infty}(\g)\otimes_{eh}L^{\infty}(\g)=\left \{\sum_i a_i\otimes b_i : ~\sum_i a_ia_i^*~\mathrm{and~} \sum_{i}b_i^*b_i~\mathrm{converges}\right\}\]
with respect to the weak $*$-topologies and the matricial norm structure comes from the completely isometric embedding into $CB^{\sigma}(B(L^2(\g)),B(L^2(\g)))$. More precisely, the matricial norm for $(\displaystyle \sum_k a^{i,j}_k\otimes b^{i,j}_k)_{i,j=1}^n\in M_n\otimes (L^{\infty}(\g)\otimes_{eh} L^{\infty}(\g))$ is given by
\[\sup_{m\in \n} \sup_{T=(T^{s,t})_{s,t=1}^m}  \left \| (\sum_k a_k^{i,j}T^{s,t}b^{i,j})_{i,j=1}^n ~_{s,t=1}^m\right \|_{M_{mn}(B(L^2(\g)))},\]
where the supremum runs over all $T=(T^{s,t})_{1\leq s,t\leq m}\in \mathrm{Ball}(M_m(B(L^2(\g))))$.

Note that the restricted comultiplication $\Delta\Bigr|_{\mathrm{Pol}(\g)}$ extends to a normal $*$-homomorphism $\Delta:L^{\infty}(\g)\rightarrow L^{\infty}(\g)\overline{\otimes}L^{\infty}(\g)$ and it induces the natural Banach algebraic structure on $L^1(\g)$. For $f_1,f_2\in L^1(\g)$, we define the convolution product of $f_1$ and $f_2$ by
\[\langle f_1*f_2,a\rangle_{L^1(\g),L^{\infty}(\g)}=(f_1\otimes f_1)(\Delta(a))~\mathrm{for~all~}a\in L^{\infty}(\g).\]

Moreover, the convolution product extends to a completely contractive map $m=\Delta_*:L^1(\g)\widehat{\otimes} L^1(\g)\rightarrow L^1(\g)$ where $\widehat{\otimes}$ is the projective tensor product in the category of operator spaces. Refer to \cite{Pi03} or \cite{EfRu00} for details. The result of \cite{Bl95}, which actually covers general completely contractive Banach algebras, is written as follows in our setting:

\begin{proposition}

The convolution algebra $L^1(\g)$ is completely isomorphic to a closed subaglebra of $B(H)$ for some Hilbert space $H$ if and only if the comultiplication $\Delta:L^{\infty}(\g)\rightarrow L^{\infty}(\g)\otimes_{eh}L^{\infty}(\g)$ is completely bounded.
\end{proposition}

\begin{lemma}\label{lem7}
Given $B=(B_{\alpha})_{\alpha\in \mathrm{Irr}(\g)}\in \mathrm{Ball}(\ell^{\infty}(\widehat{\g}))$, define an operator $T_B\in B(L^2(\g))$ by $u^{\alpha}_{j,i}\mapsto \displaystyle [Q_{\alpha}^{-1}(u^{\alpha})^*B_{\alpha}]_{j,i}=\sum_{p=1}^{n_{\alpha}}(Q_{\alpha}^{-1})_{j,j}(u^{\alpha}_{p,j})^*(B_{\alpha})_{p,i}$. Then $T_B$ is a contraction.
\end{lemma}
\begin{proof}
Take $\displaystyle v=\sum_{\alpha\in \mathrm{Irr}(\g)}\sum_{i,j=1}^{n_{\alpha}}c^{\alpha}_{i,j}u^{\alpha}_{j,i}$. Then 
\[T_Bv=\displaystyle \sum_{\alpha\in \mathrm{Irr}(\g)}\sum_{j,p=1}^{n_{\alpha}}[\sum_{i=1}^{n_{\alpha}} c^{\alpha}_{i,j}(Q_{\alpha}^{-1})_{j,j}(B_{\alpha})_{p,i}](u^{\alpha}_{p,j})^*.\]

Also, put $C_{\alpha}=(c^{\alpha}_{i,j})_{1\leq i,j\leq n_{\alpha}}$. Then 
\begin{align*}
\left \|T_Bv\right\|_{L^2(\g)}^2&=\sum_{\alpha\in \mathrm{Irr}(\g)}\sum_{j,p=1}^{n_{\alpha}}|\sum_{i=1}^{n_{\alpha}} c^{\alpha}_{i,j}(Q_{\alpha}^{-1})_{j,j}(B_{\alpha})_{p,i}|^2\frac{(Q_{\alpha})_{j,j}}{d_{\alpha}}\\
&=\sum_{\alpha\in \mathrm{Irr}(\g)}\sum_{j,p,i,i'=1}^{n_{\alpha}}(Q_{\alpha}^{-1})_{j,j}\overline{c^{\alpha}_{i,j}(B_{\alpha})_{p,i}}c^{\alpha}_{i',j}(B_{\alpha})_{p,i'}\frac{1}{d_{\alpha}}\\
&=\sum_{\alpha\in \mathrm{Irr}(\g)}\mathrm{tr}(Q_{\alpha}^{-1}C_{\alpha}^*B_{\alpha}^*B_{\alpha}C_{\alpha})\frac{1}{d_{\alpha}}\\
&\leq \sum_{\alpha\in \mathrm{Irr}(\g)}\mathrm{tr}(Q_{\alpha}^{-1}C_{\alpha}^*C_{\alpha})\frac{1}{d_{\alpha}}\\
&=\sum_{\alpha\in \mathrm{Irr}(\g)}\sum_{i,j=1}^{n_{\alpha}}(Q_{\alpha}^{-1})_{j,j}|c^{\alpha}_{i,j}|^2\frac{1}{d_{\alpha}}=\left \|v\right\|_{L^2(\g)}^2.
\end{align*}

\end{proof}

\begin{proof}[Proof of Theorem \ref{main6}]
 One direction is trivial. For the other direction, suppose that $\Delta:L^{\infty}(\g)\rightarrow L^{\infty}(\g)\otimes_{eh}L^{\infty}(\g)$ is completely bounded. Then, for any 
\[f=\sum_{\alpha\in \mathrm{Irr}(\g)}\sum_{i,j=1}^{n_{\alpha}}d_{\alpha}(\widehat{f}(\alpha)Q_{\alpha})_{i,j}u^{\alpha}_{j,i}\in \mathrm{Pol}(\mathbb{G}),\]
\[ \Delta (f)=\sum_{\alpha\in \mathrm{Irr}(\g)}\sum_{i,j=1}^{n_{\alpha}}\sum_{k=1}^{n_{\alpha}}d_{\alpha}(\widehat{f}(\alpha)Q_{\alpha})_{i,j}u^{\alpha}_{j,k}\otimes u^{\alpha}_{k,i}\in L^{\infty}(\g)\otimes_{eh}L^{\infty}(\g).\]

For each $(B_{\alpha})_{\alpha\in \mathrm{Irr}(\g)}\in \mathrm{Ball}(\ell^{\infty}(\widehat{\g}))$, pick $T_B$ of Lemma \ref{lem7}. Then $(\Delta(f))(T_B)$ sends $1_A \in L^2(\g)$ to 
\begin{align*}
x&:=\sum_{\alpha\in \mathrm{Irr}(\g)}\sum_{i,j,k,p=1}^{n_{\alpha}}d_{\alpha}(\widehat{f}(\alpha)Q_{\alpha})_{i,j}(Q_{\alpha}^{-1})_{k,k}(B_{\alpha})_{p,i}u^{\alpha}_{j,k}(u^{\alpha}_{p,k})^*\in L^2(\g).
\end{align*}

Then 
\begin{align*}
\left \|\Delta\right\|_{cb} \left \|f\right\|_{L^{\infty}(\g)}&\geq \left \|x\right\|_{L^2(\g)}\\
&\geq |h(x)|\\
&=|\sum_{\alpha\in \mathrm{Irr}(\g)}\sum_{i,j,k=1}^{n_{\alpha}}d_{\alpha}(\widehat{f}(\alpha)Q_{\alpha})_{i,j}\frac{1}{d_{\alpha}}(B_{\alpha})_{j,i}|\\
&=|\sum_{\alpha\in \mathrm{Irr}(\g)}n_{\alpha}\mathrm{tr}(\widehat{f}(\alpha)Q_{\alpha}B_{\alpha})|.
\end{align*}

Since $B_{\alpha}$ is arbitrarily chosen, we have a bounded map $\Phi_*:C_r(\g)\rightarrow \ell^1(\widehat{\g}), f\mapsto \displaystyle (\frac{n_{\alpha}}{d_{\alpha}}\widehat{f}(\alpha))_{\alpha\in \mathrm{Irr}(\g)}$ and it induces the bounded dual map $\Phi:\ell^{\infty}(\widehat{\g})\rightarrow M(\g)$ determined by
\[\langle f, \Phi(A) \rangle_{C_r(\g),M(\g)}=\sum_{\alpha\in \mathrm{Irr}(\g)}n_{\alpha}\mathrm{tr}(\widehat{f}(\alpha)Q_{\alpha}A(\alpha)). \]

Now, we can find that the map $\Psi:\ell^{\infty}(\widehat{\g})\rightarrow M(\g)$, $\displaystyle A=(A(\alpha))_{\alpha\in \mathrm{Irr}(\g)}\mapsto \mu_A\sim \sum_{\alpha\in \mathrm{Irr}(\g)}d_{\alpha}\mathrm{tr}(A_{\alpha}(\frac{n_{\alpha}}{d_{\alpha}}\mathrm{Id}_{n_{\alpha}})Q_{\alpha}u^{\alpha})$, is bounded since for any $f\in \mathrm{Pol}(\g)$ we have
\[\langle f^*,\Psi(A)\rangle_{C_r(\g),M(\g)}= \overline{\langle f,\Phi(A^*)\rangle}_{C_r(\g),M(\g)}\]
by Lemma \ref{lem1}.

Finally, Proposition \ref{lem2} and Theorem \ref{thm-nonKac} say that
\[\sum_{\alpha\in \mathrm{Irr}(\g)}\frac{d_{\alpha}}{n_{\alpha}}\mathrm{tr}(Q_{\alpha} \frac{n_{\alpha}^2}{d_{\alpha}^2}\mathrm{Id}_{n_{\alpha}})=\sum_{\alpha\in \mathrm{Irr}(\g)}n_{\alpha}<\infty,\]
so that $A$ should be finite dimensional.
\end{proof}

{\bf Acknowledgement.}  The author is grateful to Professor Hun Hee Lee for his encouragement and to Professor Gilles Pisier for his helpful comments, particularly on Remark \ref{rmk-Khintchine}.

\bibliographystyle{alpha}
\bibliography{RandomFourier}

\end{document}